\documentclass[11pt]{article}

\usepackage{amsmath,amssymb,amsthm,amscd,eucal}
\usepackage{latexsym}
\usepackage{graphicx}
\usepackage[usenames,dvipsnames,svgnames,table]{xcolor}
\usepackage{tikz}
\usepackage{color}
\usepackage{url}
 
\newtheorem{thm}[equation]{Theorem}
\newtheorem{cor}[equation]{Corollary}
\newtheorem{lemma}[equation]{Lemma}
\newtheorem{prop}[equation]{Proposition}

\newtheorem*{conj*}{Conjecture}

\theoremstyle{definition}

\newtheorem{remark}[equation]{Remark}
\newtheorem{remarks}[equation]{Remarks}

\newtheorem{exams}[equation]{Examples}

\numberwithin{equation}{section}

\renewcommand{\ker}{\mathsf{ker}}

\newcommand{\ZZ}{\mathbb{Z}}  
\newcommand{\CC}{\mathbb{C}}

\newcommand{\A}{\mathsf{A}}

\newcommand{\Hf}{\mathcal{H}(f)} 

\def\aut{ \mathsf {Aut_\CC}(\Hf)} 

\def\cent1{\mathsf{Z}(\A_1)}

\def\degg{\, \mathsf{deg} \,}

\usepackage[letterpaper,margin=1.38in]{geometry}

\begin{document}  
\title{Non-Noetherian generalized Heisenberg algebras}
\author{Samuel A.\ Lopes\thanks{The author was partially supported by CMUP (UID/MAT/00144/2013), which is funded by FCT (Portugal) with national (MEC) and European structural funds through the programs FEDER, under the partnership agreement PT2020.}}

\date{}
\maketitle

\vspace{-.25 truein}  
\begin{abstract} 
In this note we classify the non-Noetherian generalized Heisenberg algebras $\Hf$ introduced in \cite{LZ15}. In case $\degg f>1$, we determine all locally finite and also all locally nilpotent derivations of $\Hf$ and describe the automorphism group of these algebras.
\end{abstract}   

\noindent 
\textbf{MSC Numbers (2010)}: Primary 16W25, 16W20; Secondary 16P40, 16S36.\hfill \newline

\maketitle

\section{Introduction}\label{S:intro}

Fix a polynomial $f\in\CC[h]$. The \textit{generalized Heisenberg algebra} $\Hf$ is the unital associative $\CC$-algebra with generators $x$, $y$, $h$ satisfying the relations:
\begin{equation}\label{E:intro:def}
hx=xf(h), \quad yh=f(h)y, \quad yx-xy=f(h)-h.
\end{equation}
See \cite{LZ15} and the references therein for information on how these algebras first appeared and on their applications to theoretical physics.

Ambiskew polynomial rings were introduced by Jordan over a series of papers (see the references in~\cite{dJ00}), but for our purposes the best suited definition is the one found in \cite{dJ00}, which we briefly recall. Let $\sigma$ be an endomorphism of a commutative $\CC$-algebra $B$, $c\in B$ and $p\in\CC$. The ambiskew polynomial ring $R(B, \sigma, c, p)$ is the $\CC$-algebra generated by $B$ and two indeterminates, $x$ and $y$, subject to the relations 
\begin{equation*}
bx=x\sigma(b),\quad yb=\sigma(b)y, \quad yx-pxy=c,\quad \mbox{for all $b\in B$.}
\end{equation*}
On comparing these relations with those in \eqref{E:intro:def}, one immediately sees that 
\begin{equation}\label{E:intro:ambi}
\Hf\cong R(\CC[h], \sigma, f(h)-h, 1),
\end{equation}
where $\sigma:\CC[h]\rightarrow\CC[h]$ is the algebra endomorphism given by $\sigma(h)=f(h)$. In particular, one can see that there is an overlap between the generalized Heisenberg algebras defined above and (generalized) down-up algebras (see Corollary~\ref{C:n:gdua} below).

The algebras $\Hf$ can also be seen as weak generalized Weyl algebras over a polynomial algebra in two variables, in the sense of \cite{LMZ15}, a construction which includes the generalized Weyl algebras introduced by V.V.\ Bavula in \cite{vB92}.  
In \cite{LZ15} the authors determine a basis for $\Hf$ over $\CC$, compute the center of $\Hf$, solve the isomorphism problem for this family of algebras and classify all the finite-dimensional irreducible representations of $\Hf$.

In this note we show that $\Hf$ is (right or left) Noetherian if and only if $\degg f=1$ and that $\Hf$ is isomorphic to a generalized down-up algebra if and only if $\degg f\leq 1$. For this reason, we then concentrate on the case where $\degg f>1$ and determine  the locally nilpotent and the locally finite derivations of $\Hf$, all $\ZZ$-gradings of $\Hf$ and describe the automorphism group of $\Hf$. In particular, we obtain the following results in case $\degg f>1$:
\begin{itemize}
\item[(i)]  $\Hf$ in neither right nor left Noetherian (Proposition~\ref{P:n:n});
\item[(ii)] $\Hf$ admits a unique (up to an integer multiple) nontrivial $\ZZ$-grading, in which $x$ has degree $1$, $y$ has degree $-1$ and $h$ has degree $0$ (Corollary~\ref{C:lfd:zgrad});
\item[(iii)] the automorphism group of $\Hf$ is abelian: it is isomorphic to $\CC^*\times \mathsf{C}$, where $\mathsf{C}$ is a finite cyclic group whose order divides $(\degg f) -1$ (Theorem~\ref{T:aut:main}).
\end{itemize}
 
In Section~\ref{S:n} of the paper we review some properties of $\Hf$ which have been established in \cite{LZ15}, determine when $\Hf$ is Noetherian and when it is isomorphic to a generalized down-up algebra, while in Section~\ref{S:lfd:sg} we introduce a useful commutative subalgebra of $\Hf$, which is a maximal commutative subalgebra if $\degg f>1$. Assuming that $\degg f>1$, we then investigate the locally finite and the locally nilpotent derivations of $\Hf$ and also its $\ZZ$-gradings in Section~\ref{S:lfd}, and in the final section, Section~\ref{S:aut}, we describe the automorphism group of $\Hf$ and show that it is always an abelian group generated by the automorphisms which fix $h$ and the automorphisms which fix $x$.

We make use of the commutator notation $[a, b]=ab-ba$. The sets of integers, nonnegative integers and positive integers are denoted by $\ZZ$, $\ZZ_{\geq 0}$ and $\ZZ_{> 0}$, respectively. The field of complex numbers is denoted by $\CC$, and the multiplicative group of nonzero complex numbers is denoted by $\CC^*$. For a polynomial $g\in\CC[h]$, $\degg g$ will always denote the degree of $g$ as a polynomial in $h$. 

Throughout the paper, $\sigma:\CC[h]\rightarrow\CC[h]$ is the algebra endomorphism given by $\sigma(h)=f(h)$. For any function $\phi: X\rightarrow X$, we will use the notation $\phi^k$ to mean the $k$-th power of $\phi$ with respect to composition. In particular, $\phi^0$ denotes the identity on the set $X$.

\section{The Noetherian property}\label{S:n}

Below we record a few results from \cite{LZ15} which will be useful in the course of this paper. As usual, $Z(\Hf)$ denotes the center of $\Hf$.

\begin{lemma}[{\cite[Lemma 1, Lemma 2, Theorem 4]{LZ15}}]  \label{L:n:LZ} 
Let $f\in\CC[h]$. Then:
\begin{itemize}
\item[\textup{(a)}] The set $\{ x^i h^j y^k \mid i, j, k\in \ZZ_{\geq 0}  \}$ is a basis of $\Hf$. 
\item[\textup{(b)}] The algebra $\Hf$ is a domain if and only if $\degg f\geq 1$.
\item[\textup{(c)}] The center of $\Hf$ contains the polynomial algebra $\CC[z]$, where $z=xy-h=yx-f(h)$. If $\degg f\neq 1$, then $Z(\Hf)=\CC[z]$.
\end{itemize}
\end{lemma}

\begin{remarks}\hfill
\begin{enumerate}
\item Identifying $\Hf$ with the ambiskew polynomial ring $R(\CC[h], \sigma, f(h)-h, 1)$ as in \eqref{E:intro:ambi}, it follows that $\Hf$ is conformal, as defined in \cite[Section 2.3]{dJ00}, and the corresponding Casimir element is precisely the central element $z=xy-h$ defined above.
\item Suppose $f\in\CC$. Then by considering the generators $-x$, $y$ and $h-f$, we see that $\Hf\cong R(\CC[h], \sigma, h, 1)$, with $\sigma=0$, and from \cite[Theorem 7.10]{dJ00} we conclude that $\Hf$ is a prime ring. Thus by Lemma~\ref{L:n:LZ}(b), $\Hf$ is a prime ring for any $f\in\CC[h]$. 
\item Since the center of $\Hf$ contains the polynomial algebra $\CC[z]$ and $\Hf$ has countable dimension over $\CC$, it follows from Dixmier's version of Schur's Lemma that $\Hf$ is never primitive.
\end{enumerate}
\end{remarks}

There is an order two anti-automorphism of $\Hf$, denoted by $\iota$, that fixes $h$ and interchanges $x$ and $y$: 
\begin{equation}\label{E:iota}
\iota: \Hf\rightarrow\Hf, \quad\quad x\mapsto y, \quad y\mapsto x, \quad h\mapsto h.
\end{equation}
Hence $\Hf$ is isomorphic to its opposite algebra $\Hf^{\text{op}}$.

\begin{prop}\label{P:n:n} 
The algebra $\Hf$ is right (respectively, left) Noetherian if and only if $\degg f=1$.
\end{prop}
\begin{proof}
If $\degg f=1$ then $\Hf$ is a generalized Weyl algebra over a polynomial ring in two variables, and thus it is right and left Noetherian. So assume that $\degg f\neq 1$. In particular, $f(h)-h$ has some root $\alpha\in\CC$. Let $F(h)=f(h+\alpha)-\alpha$. Then $\degg F=\degg f$ (here we assume the zero polynomial has degree $0$) and $F(h)\in h\CC[h]$. Moreover, $F(h-\alpha)=f(h)-\alpha$ and then $\Hf\cong \mathcal{H}(F)$ by \cite[Lemma 3]{LZ15}. So there is no loss in assuming that $f(h)\in h\CC[h]$. By the isomorphism $\Hf\cong \Hf^{\text{op}}$ it will be enough to show that $\Hf$ is not left Noetherian.

For each $n\in\ZZ_{\geq 0}$ define the left ideal 
\begin{equation*}
I_n=\sum_{i=0}^n \Hf hy^i.
\end{equation*}
Then $I_n\subseteq I_{n+1}$ for all $n\geq 0$ and we finish the proof by showing that these inclusions are strict. Note that by Lemma~\ref{L:n:LZ}(a),
\begin{equation*}
\Hf=\bigoplus_{j, k\geq 0}x^j\CC[h]y^k.
\end{equation*}
Given $j, k\geq 0$ and $g(h)\in\CC[h]$, we have $x^j g(h)y^k hy^i=x^j g(h)\sigma^k(h)y^{k+i}$. Assume, by way of contradiction, that $hy^{n+1}\in I_n$. Then there exist $g_i(h)\in\CC[h]$, $i=0, \ldots, n$, such that $\displaystyle hy^{n+1}=\sum_{i=0}^n g_i (h)\sigma^{n+1-i}(h)y^{n+1}$. It follows by Lemma~\ref{L:n:LZ}(a)  that
\begin{equation}\label{E:P:n:n}
h=\sum_{i=0}^n g_i (h)\sigma^{n+1-i}(h).
\end{equation}
As by hypothesis $\sigma(h)=f(h)\in h\CC[h]$, one can deduce that $\sigma^{n+1-i}(h)\in f(h)\CC[h]$ for all $0\leq i\leq n$ and  \eqref{E:P:n:n}
then implies that $h\in f(h)\CC[h]$, which is a contradiction since under our hypothesis either $f(h)=0$ or $\degg f>1$. This proves that $hy^{n+1}\not\in I_n$ for any $n\geq 0$ and hence $\left\{ I_n \right\}_{n\geq 0}$ is a strict ascending chain of left ideals of $\Hf$.
\end{proof}

\begin{remark}
The case $f\in\CC$ of Proposition~\ref{P:n:n} follows also from  \cite[Corollary 7.3]{dJ00}, which applies when $\sigma$ is not injective. In terms of the endomorphism $\sigma$, Proposition~\ref{P:n:n} could be restated as: The algebra $\Hf$ is right (respectively, left) Noetherian if and only if $\sigma$ is an automorphism.
\end{remark}

We recall that a generalized down-up algebra $L(g, r, s, \gamma)$, given by the parameters $g\in\CC[H]$ and $r, s, \gamma\in\CC$, is defined as the unital associative $\CC$-algebra generated by $d$, $u$ and $H$, subject to the relations:
\begin{equation*}
dH-rHd+\gamma d=0,\quad
Hu-ruH+\gamma u=0,\quad
du-sud+g(H)=0.
\end{equation*}
Generalized down-up algebras were defined in~\cite{CS04} as generalizations of the down-up algebras introduced by Benkart and Roby in \cite{BR98}. Generalized down-up algebras include all down-up algebras, the algebras similar to the enveloping algebra of $\mathfrak{sl}_{2}$ defined by Smith~\cite{spS90}, Le Bruyn's conformal $\mathfrak{sl}_{2}$ enveloping algebras~\cite{lLB95} and Rueda's algebras similar to the enveloping algebra of $\mathfrak{sl}_{2}$~\cite{sR02}.

\begin{cor}\label{C:n:gdua}
The algebra $\Hf$ is isomorphic to a generalized down-up algebra if and only if $\degg f\leq 1$.
\end{cor}
\begin{proof}
Suppose first that $\degg f\leq 1$, say $f(h)=ah+b$ for $a, b\in\CC$. Then it is straightforward to verify that $\Hf\cong L(H-f(H), a, 1, -b)$, under an isomorphism that sends $x$, $y$ and $h$ to $u$, $d$ and $H$, respectively. Conversely, suppose that $\degg f>1$. Then by Proposition~\ref{P:n:n} and Lemma~\ref{L:n:LZ}(b), $\Hf$ is a non-Noetherian domain. Hence $\Hf$ cannot be isomorphic to a generalized down-up algebra, as a generalized down-up algebra is a domain if and only if it is Noetherian, by Propositions 2.5 and 2.6 of \cite{CS04}.
\end{proof}

In view of this result, we will henceforth focus most of our attention on the generalized Heisenberg algebras $\Hf$ with $f\in\CC[h]$ such that $\degg f>1$.

\section{The commutative algebra $\Hf_0$}\label{S:lfd:sg}

In this short section we record a few useful formulas for computing in $\Hf$ and then explore an interesting commutative subalgebra of $\Hf$.

\begin{lemma}\label{L:c:c} Let $k\in\ZZ_{\geq 0}$ and $g\in\CC[h]$. Then the following hold:
\begin{itemize}
\item[\textup{(a)}] $[y, x^k]=x^{k-1}(\sigma^k (h)-h)$;
\item[\textup{(b)}] $[y^k, x]=(\sigma^k (h)-h)y^{k-1}$;
\item[\textup{(c)}] $(x^k g y^k)x=x(x^k \sigma(g) y^k + x^{k-1} (\sigma^k (h)-h)g y^{k-1})$;
\item[\textup{(d)}] $y(x^k g y^k)=(x^k \sigma(g) y^k + x^{k-1} (\sigma^k (h)-h)g y^{k-1})y$;
\item[\textup{(e)}] $x^k g y^k$ commutes with $x^j \tilde g y^j$ for all $\tilde g\in\CC[h]$ and all $j\in\ZZ_{\geq 0}$.
\end{itemize}
\end{lemma}
\begin{proof}
Parts (a) and (b) have been established in \cite{dJ00}, formulas (6a)--(6b). We prove part (c) using (b):
\begin{align*}
(x^k g y^k)x&=x^k g xy^k +x^k g [y^k,x] \\
&=x^{k+1} \sigma(g) y^k +x^k g (\sigma^k (h)-h)y^{k-1}\\
&=x(x^{k} \sigma(g) y^k +x^{k-1} (\sigma^k (h)-h)g y^{k-1}).
\end{align*}
Formula (d) follows from applying the anti-automorphism $\iota$ of \eqref{E:iota} to (c).

Finally, we prove (e) by induction on $k$, the case $k=0$ being trivial:
\begin{equation*}
g(x^j \tilde g y^j)=x^j \sigma^j(g)\tilde g y^j=x^j \tilde g \sigma^j(g) y^j=(x^j \tilde g  y^j) g. 
\end{equation*}
Now suppose (e) holds for a certain $k\geq 0$. Thus we have:
\begin{align*}
 (x^{k+1} g y^{k+1})(x^j \tilde g y^j) &=  (x^{k+1} g y^{k})y(x^j \tilde g y^j)\\
 &=  x(x^{k} g y^{k})(x^j \sigma(\tilde g) y^j + x^{j-1} (\sigma^j (h)-h)\tilde g y^{j-1})y \quad\quad \mbox{by (d)}\\
 &=  x(x^j \sigma(\tilde g) y^j + x^{j-1} (\sigma^j (h)-h)\tilde g y^{j-1})(x^{k} g y^{k})y \quad\quad \mbox{($\ast$)}\\
 &=  (x^j \tilde g y^j)x(x^{k} g y^{k})y \quad\quad \mbox{by (c)}\\
 &=  (x^j \tilde g y^j)(x^{k+1} g y^{k+1}),
\end{align*}
where ($\ast$) follows from the induction hypothesis. So (e) holds for all $k\in\ZZ_{\geq 0}$.
\end{proof}

There is an obvious grading of $\Hf$ relative to which $x$ has degree $1$, $y$ has degree $-1$ and $h$ has degree $0$. We denote the corresponding homogeneous subspaces by $\Hf_\ell$, for $\ell\in\ZZ$, so that 
\begin{equation}\label{E:zgrading}
\Hf=\bigoplus_{\ell\in\ZZ} \Hf_\ell, \quad \quad \mbox{with \quad $\displaystyle \Hf_\ell=\bigoplus_{i-k=\ell}\CC\, x^i \CC[h] y^k$.}
\end{equation}
We call this the \textit{standard grading} of $\Hf$, and, whenever we mention a homogeneous component or element of $\Hf$, we will always be referring to this standard grading. 

The subalgebra $\Hf_0$ has basis $\{x^k h^j y^k \mid k, j\geq 0\}$ and 
$\Hf_\ell=x^\ell \, \Hf_0$ if $\ell\geq 0$; $\Hf_\ell= \Hf_0 \, y^{-\ell}$ if $\ell\leq 0$. Thus we have the decomposition
\begin{equation*}
\Hf=\bigoplus_{\ell>0} x^\ell\,\Hf_0 \oplus \Hf_0 \oplus \bigoplus_{\ell>0} \Hf_0 \, y^\ell.
\end{equation*}

\begin{prop}\label{P:lfd:sg:hoiscom}
The subalgebra $\Hf_0$ is commutative. If $\degg f>1$, then $\Hf_0$ is a maximal commutative subalgebra of $\Hf$ which strictly contains $\CC[z, h]$, the polynomial subalgebra of $\Hf$ generated by $h$ and the central element $z=xy-h$.
\end{prop}
 
\begin{proof}
The first statement is a direct consequence of Lemma~\ref{L:c:c}(e). Assume now that $\degg f>1$. Then $\sigma$ is injective and has infinite order. For any $i, k\in \ZZ_{\geq 0}$ and $g\in\CC[h]$,  $[h, x^i g y^k]=x^i (\sigma^i(h)-\sigma^k(h))g y^k$. Hence, if $g\neq 0$, we deduce that $[h, x^i g y^k]=0\iff i=k$, and from this it is straightforward to conclude that $\Hf_0$ is the centralizer of $h$, hence a maximal commutative subalgebra of $\Hf$.

The commuting elements $h$ and $z$ are homogeneous of degree $0$ and are easily seen to be algebraically independent, as $z^k-x^ky^k$ is in the span of $\{ x^i g y^i \mid i<k,\ g\in\CC[h]\}$. Suppose, by contradiction, that there exist $g_k\in\CC[h]$ such that $xhy=\sum_{k\geq 0} g_k z^k$. Then by the argument above we must have $g_k=0$ for all $k>1$ and $\sigma(g_1)=h$, which is possible only if $\degg f=1$. Therefore $xhy\in\Hf_0\setminus \CC[z, h]$.
\end{proof}

By Lemma~\ref{L:c:c}(c)--(d), it is possible to extend $\sigma$ to a $\CC$-linear endomorphism $\tilde\sigma$ of $\Hf_0$ so that $\tilde\sigma(x^k g y^k)=x^k \sigma(g) y^k + x^{k-1} (\sigma^k (h)-h)g y^{k-1}$, for all $k\in\ZZ_{\geq 0}$ and $g\in\CC[h]$. For simplicity, we still denote this endomorphism by $\sigma$ instead of $\tilde\sigma$. By Lemma~\ref{L:c:c}(c)--(d) and Lemma~\ref{L:n:LZ}(a), $\sigma$ is defined by the relations:
\begin{equation}\label{E:lfd:sg:sigma}
\theta x=x \sigma (\theta), \quad\quad y\theta=\sigma(\theta)y, \quad \quad \mbox{for all $\theta\in\Hf_0$.}
\end{equation}
In particular, \eqref{E:lfd:sg:sigma} implies that $\sigma$ is an algebra endomorphism of $\Hf_0$.

\section{locally finite derivations of $\Hf$ when $\degg f>1$}\label{S:lfd}

Henceforth we will assume that $\degg f>1$. By Corollary~\ref{C:n:gdua} we are assuming that $\Hf$ is not a generalized down-up algebra. Most of our subsequent results do not hold if $\degg f\leq 1$.

Our goal in this section is to determine all locally finite derivations of $\Hf$. In particular, we will classify all $\ZZ$-gradings of $\Hf$ and show that $\Hf$ has no nontrivial locally nilpotent derivations. Our methods are akin to those used in \cite{SAV15}.

Let $\delta$ be a $\CC$-linear endomorphism of $\Hf$. We recall the following standard definitions:
\begin{itemize}
\item $\delta$ is a derivation of $\Hf$ if $\delta(ab)=\delta(a)b+a\delta(b)$;
\item $\delta$ is locally finite if for every $a\in\Hf$ the $\CC$-linear span of $\{ \delta^k(a) \mid k\in\ZZ_{\geq 0}\}$ is finite dimensional;
\item $\delta$ is locally nilpotent if for every $a\in\Hf$ there is $k\in\ZZ_{\geq 0}$ such that $\delta^k(a)=0$;
\item $\delta$ is homogeneous of degree $r\in\ZZ$ if $\delta \left(\Hf_\ell\right)\subseteq \Hf_{\ell+r}$ for all $\ell\in\ZZ$.
\end{itemize}

Assume $\delta$ is any derivation of $\Hf$. Since $\Hf$ is finitely generated, there exist homogeneous derivations $\delta_1, \ldots, \delta_k$ of strictly increasing degrees such that $\delta=\delta_1+\cdots +\delta_k$. Moreover, as seen in \cite[Lemma 1.1]{SAV15}, if $\delta$ is locally finite, then so are $\delta_1$ and $\delta_k$, and if $\delta_1$ (respectively, $\delta_k$) is of nonzero degree, then it must be locally nilpotent. 

We need one final definition. Given a locally nilpotent derivation $\delta$ and $a\in\Hf$, define
\begin{equation*}
\deg_\delta (a)=\max \{ k\in\ZZ_{\geq 0} \mid \delta^{k}(a)\neq 0\} \quad \mbox{if $a\neq 0$;}
\end{equation*}
define also $\deg_\delta(0)=-\infty$. It can be easily checked (see for example \cite{lMLnotes}) that for $a, b\in\Hf$, $\deg_\delta(a+b)\leq \max\{\deg_\delta(a), \deg_\delta(b)\}$, with equality if $\deg_\delta(a)\neq\deg_\delta(b)$. Since $\Hf$ is a domain and $\CC$ has characteristic $0$, we also have, from the Leibniz rule, $\deg_\delta(ab)= \deg_\delta(a)+\deg_\delta(b)$. In particular, $\ker\delta$ is factorially closed: if $\delta(ab)=0$ for some nonzero $a, b\in\Hf$, then $\delta(a)=0=\delta(b)$.

\begin{lemma}\label{L:lfd:main}
Assume that $\degg f>1$. Then all locally finite derivations of $\Hf$ are homogeneous of degree $0$. 
\end{lemma}

\begin{proof}
Let $\delta$ be a locally finite derivation of $\Hf$. By the above decomposition $\delta=\delta_1+\cdots +\delta_k$ of $\delta$ into homogeneous derivations of strictly increasing degrees, it will be enough to show that there are no nonzero homogeneous locally nilpotent derivations of $\Hf$ of degree $r\neq 0$. 

So assume $\delta$ is a homogeneous locally nilpotent derivation of $\Hf$ of degree $r\neq 0$. Let $d=\deg_\delta(h)$ and suppose that $d>0$. Then $\deg_\delta (f(h))=d\degg f$ and  the relation $hx=xf(h)$ yields
\begin{equation*}
d+\deg_\delta(x)=\deg_\delta(x)+d\degg f,
\end{equation*}
so $\degg f=1$, which contradicts our hypothesis. Hence $d=0$ and $\delta(h)=0$.

By replacing $\delta$ with $\iota\delta\iota^{-1}$, where $\iota$ is the anti-automorphism defined in \eqref{E:iota}, we can assume that $r>0$. Then $\ker\delta$ contains some nonzero homogeneous element of positive degree. Since elements of $\Hf$ of positive degree lie in $x\Hf$ and $\ker\delta$ is factorially closed, we deduce that $\delta(x)=0$.

Any derivation maps the center of an algebra into itself, so $\delta$ restricts to a locally nilpotent derivation of $\CC[z]$, by Lemma~\ref{L:n:LZ}(c), and thus $\delta(z)\in\CC$. On the other hand, since $z=xy-h$ is homogeneous of degree $0$ and $\delta$ has positive degree, it must be that $\delta(z)=0$, and from $0=\delta(z)=x\delta(y)$, we conclude that $\delta(y)=0$. Then $\delta=0$ and the lemma is proved.
\end{proof}

The next theorem, our main result on derivations of $\Hf$ when $\degg f>1$, shows that the space of locally finite derivations of $\Hf$ is one-dimensional over $\CC$, spanned by the derivation $\partial$ defined by
\begin{equation}\label{E:lfd:partial}
\partial(x^i h^j y^k)=(i-k)x^i h^j y^k,\quad \mbox{for all $i, j, k\in \ZZ_{\geq 0}$.}
\end{equation}

\begin{thm}\label{T:lfd:main}
Assume that $\degg f>1$.  If $\delta$ is a locally finite derivation of $\Hf$, then there is $\lambda\in\CC$ such that $\delta(x)=\lambda x$, $\delta(y)=-\lambda y$ and $\delta(h)=0$.
\end{thm}
\begin{proof}
Let $\delta$ be a locally finite derivation of $\Hf$. By Lemma~\ref{L:lfd:main}, we know that $\delta$ is homogeneous of degree $0$, so there are $\theta_x, \theta_h, \theta_y\in\Hf_0$ so that
\begin{equation*}
\delta(x)=x\theta_x, \quad  \delta (h)=\theta_h, \quad\mbox{and}\quad \delta(y)=\theta_y y.
\end{equation*}
In particular, since $h$ commutes with $\theta_h$, we have $\delta(g(h))=g'(h)\theta_h$ for all $g(h)\in\CC[h]$, where $g'(h)$ denotes the derivative of $g(h)$ with respect to $h$. 

\medskip 

\textit{Claim 1:} {\it $\theta_h=0$ and $\theta_x+\theta_y=0$.}

\smallskip
\textit{Proof of Claim 1:}  Write 
\begin{equation}\label{E:lfd:main:dt2}
\theta_h=\sum_{k\geq 0}x^kg_k(h) y^k,
\end{equation}
with $g_k(h)\in\CC[h]$ and $g_k(h)=0$ except for finitely many indices $k$.

As observed in the proof of Lemma~\ref{L:lfd:main}, $\delta$ restricts to a locally finite derivation of $\CC[z]$, the center of $\Hf$, and thus $\delta(z)\in\CC\oplus\CC z$, say $\delta (z)=\mu z-\lambda$, with $\lambda, \mu\in\CC$. Since $\mu z-\lambda=\delta (xy-h)=x(\theta_x+\theta_y)y-\theta_h$, we have 
\begin{equation}\label{E:lfd:main:z}
\theta_h=x(\theta_x+\theta_y)y-\mu z+\lambda=x(\theta_x+\theta_y -\mu)y+\mu h+\lambda.
\end{equation}
In particular, $g_0(h)=\mu h+\lambda$.

We now apply $\delta$ to the relation $yh=f(h)y$ and get $\theta_y yh +y\theta_h=f'(h)\theta_h y+f(h)\theta_y y$. As $h$ and $\theta_y$ commute, and $y\theta_h=\sigma(\theta_h)y$, by \eqref{E:lfd:sg:sigma}, we obtain
\begin{equation}\label{E:lfd:main:t2comm}
\sigma(\theta_h)=f'(h)\theta_h .
\end{equation}
Now combining \eqref{E:lfd:main:dt2} and \eqref{E:lfd:main:t2comm} we deduce that, for every $k\geq 0$:
\begin{equation}\label{E:lfd:main:t2rels}
\sigma^k(f'(h))g_k(h)=\sigma(g_k(h)) + (\sigma^{k+1} (h)-h)g_{k+1}(h).
\end{equation}
Setting $k=0$ in \eqref{E:lfd:main:t2rels} we obtain $(f(h)-h)g_{1}(h)=f'(h)g_0(h)-\sigma(g_0(h))$. Since we have already established that $\degg g_0\leq 1$, we deduce now from the latter equation that $\degg (f(h)-h)g_{1}(h)\leq \degg f$, and thus $g_1\in\CC$. Combining this with the $k=1$ case of \eqref{E:lfd:main:t2rels}, $\sigma(f'(h))g_1(h)=\sigma(g_1(h)) + (\sigma (f(h))-h)g_{2}(h)$, yields $g_2=0$, and in turn the latter gives $g_k=0$ for all $k\geq 2$. Using again the relation $\sigma(f'(h))g_1(h)=\sigma(g_1(h)) + (\sigma (f(h))-h)g_{2}(h)$ with $g_2=0$ and $g_1\in\CC$ gives $g_1=0$. Therefore we have
\begin{equation}\label{E:lfd:main:g0}
\sigma (g_0)=f'(h)g_0.
\end{equation}
Suppose $g_0\neq 0$, and let $a$ be the leading coefficient of $f(h)$. Then $\mu\neq 0$  and comparing leading coefficients in \eqref{E:lfd:main:g0} yields $\mu a=a(\degg f)\mu$, whence $\degg f=1$, which is a contradiction. Thus $g_0=0$. 

From the above we conclude that $\theta_h=\sum_{k\geq 0}x^kg_k(h) y^k=0$ and finally by \eqref{E:lfd:main:z} we get $\theta_x+\theta_y=0$, establishing Claim 1.

\medskip

So far we have shown that
\begin{equation*}
\delta(x)=x\theta_x, \quad  \delta (h)=0, \quad\mbox{and}\quad \delta(y)=-\theta_x y,
\end{equation*}
so it remains to be inferred that $\theta_x\in\CC$. 

\medskip 

\textit{Claim 2:} {\it $\delta(\theta)=0$, for all $\theta\in\Hf_0$.}

\smallskip
\textit{Proof of Claim 2:}  Since $\delta(g)=0$ for all $g\in\CC[h]$, it suffices to show that if $\theta\in\Hf_0$ and $\delta (\theta)=0$, then also $\delta (x\theta y)=0$. This follows easily using the fact that $\Hf_0$ is commutative, as proved in Proposition~\ref{P:lfd:sg:hoiscom}.

\medskip

From Claim 2 it follows that, for all $k\geq 0$, $\delta (\theta_x^k)=0$, which implies that $\delta^k(x)=x\theta_x^k$. As $\delta$ is locally finite, the span of $\{ \theta_x^k \mid k\in\ZZ_{\geq 0}\}$ must then be finite dimensional. This is possible only if $\theta_x\in\CC$, thus finishing the proof of the theorem.
\end{proof}

Since locally nilotent derivations are locally finite, we derive the following corollary.

\begin{cor}
Assume that $\degg f>1$. Then $\Hf$ has no nonzero locally nilpotent derivations. 
\end{cor}

Suppose that $\Hf=\bigoplus_{\alpha\in \CC} V_\alpha$ is a grading. Define the $\CC$-linear endomorphism $\delta$ of $\Hf$ by $\delta (v_\alpha)=\alpha v_\alpha$ for all $v_\alpha \in V_\alpha$ and all $\alpha\in \CC$. It is immediate to check that $\delta$ is a diagonalizable derivation of $\Hf$ whose eigenvalues are those $\alpha\in \CC$ such that $V_\alpha\neq (0)$. Conversely, if $\delta$ is a diagonalizable derivation, then $\delta$ determines a grading where $V_\alpha$ is the $\alpha$-eigenspace of $\delta$. Furthermore, diagonalizable derivations are clearly locally finite.

Thus, we deduce from Theorem~\ref{T:lfd:main} that, except for the trivial grading in which every element of $\Hf$ has degree $0$, $\Hf$ only admits the standard grading defined in \eqref{E:zgrading}, up to scaling by some integer. More precisely, we have:

\begin{cor}\label{C:lfd:zgrad}
Assume that $\degg f>1$. Then for any $\ZZ$-grading of $\Hf$, there is an integer $\ell\in\ZZ$ such that, relative to that grading, $x$ has degree $\ell$, $y$ has degree $-\ell$ and $h$ has degree $0$.
\end{cor}

\section{Automorphisms of $\Hf$ when $\degg f>1$}\label{S:aut}

When $\deg f=1$ the algebra $\Hf$ is a Noetherian generalized down-up algebra, by Corollary~\ref{C:n:gdua}, and the automorphisms of the latter have been investigated in~\cite{CL09}. We continue to assume that $\degg f>1$ and note again that our results do not generalize to the cases with $\degg f\leq 1$.

Since $\Hf$ has no nonzero locally nilpotent derivations, it seems natural to conjecture that the automorphism group of $\Hf$ is somewhat small. However, over $\CC$ we can consider also the exponential of a diagonalizable derivation. Specifically, let $c\in\CC$ and let $\partial$ be the derivation of $\Hf$ defined in \eqref{E:lfd:partial}. Then the expression
\begin{equation*}
\exp(c\partial) := \sum_{k=0}^\infty  \frac{(c\partial)^k}{k!}
\end{equation*}
defines an automorphism of $\Hf$ satisfying
\begin{equation*}
\exp(c\partial) (x) = \sum_{k=0}^\infty  \frac{c^k}{k!}x=\exp(c)x, \quad \exp(c\partial) (y)=\exp(-c)y, \quad \exp(c\partial) (h) = h,
\end{equation*}
with inverse $\exp(-c\partial)$.

The above motivates the following definition. For each $\lambda\in\CC^*$, let $\phi_\lambda$ be the automorphism of $\Hf$ defined by
\begin{equation}\label{E:aut:defpl}
\phi_\lambda(x)=\lambda x, \quad \phi_\lambda(y)=\lambda^{-1} y, \quad \phi_\lambda(h)=h.
\end{equation}

The group of algebra automorphisms of $\Hf$ will be denoted by $\aut$. We have a first description of $\aut$ below.

\begin{prop}\label{P:aut:H}
Assume $\degg f> 1$. Then the following hold:
\begin{itemize}
\item[\textup{(a)}] Any automorphism of $\Hf$ restricts to an automorphisms of $\CC[h]$, and $x$ and $y$ are eigenvectors.
\item[\textup{(b)}] $\displaystyle \{\phi\in\aut \mid \phi(h)=h\}=\{\phi_\lambda\mid \lambda\in\CC^*\}\cong\CC^*$, and this is a central subgroup of $\aut$.
\item[\textup{(c)}] $\displaystyle \{\phi\in\aut \mid \phi(x)=x\}$ is a finite cyclic subgroup whose order divides $(\degg f)-1$.
\end{itemize}
\end{prop}
\begin{proof}
Let $\phi$ be an automorphism of $\Hf$. Then as argued in Claim 4 of the proof of \cite[Theorem 5]{LZ15}, the relation $\phi(h)\phi(x)=\phi(x)f(\phi(h))$ with $\degg f > 1$ implies that $\phi(h)\in\CC[h]$; applying this result to $\phi^{-1}$ gives that $\phi(h)=ah+b$, for some $a, b\in\CC$ with $a\neq 0$.

Now writing $\phi(x)$ as a sum of terms of the form $x^i g_{i, j}y^j$ with $i, j\in\ZZ_{\geq 0}$ and $g_{i, j}\in\CC[h]$, and comparing the corresponding expressions for $\phi(h)\phi(x)$ and $\phi(x)f(\phi(h))$, we obtain $\phi(x)\in\Hf_1$. Similarly, $\phi(y)\in\Hf_{-1}$, so $\phi$ is homogeneous of degree $0$. Thus, there exist $\theta_x, \theta_y\in\Hf_0$ such that $\phi(x)=x\theta_x$ and $\phi(y)=\theta_y y$. Applying the same reasoning to $\phi^{-1}$, we deduce that $\theta_x, \theta_y\in\CC^*$, which proves (a).

Now assume $\phi\in\aut$ and $\phi(h)=h$. By (a) there exist $\lambda, \mu\in\CC^*$ such that $\phi(x)=\lambda x$ and $\phi(y)=\mu y$. Applying $\phi$ to the relation $[y, x]=f(h)-h$ yields $\lambda\mu=1$, so $\phi=\phi_\lambda$. This proves the equality in (b), and the isomorphism $\{\phi_\lambda\mid \lambda\in\CC^*\}\cong\CC^*$ is clear, as $\phi_\lambda\circ\phi_\mu=\phi_{\lambda\mu}$ for all $\lambda, \mu\in\CC^*$. 

Next, we show that the subgroup $\{\phi_\lambda\mid \lambda\in\CC^*\}$ is central in $\aut$. Let $\lambda\in\CC^*$, and suppose $\psi\in\aut$ is arbitrary. By (a) we know that $\psi(h)\in\CC[h]$, which implies that $\phi_\lambda\circ\psi(h)=\psi\circ\phi_\lambda(h)$. But as $x$ and $y$ are eigenvalues for any automorphism of $\Hf$, $\phi_\lambda\circ\psi$ and $\psi\circ\phi_\lambda$ also agree on these generators, and thus $\phi_\lambda\circ\psi=\psi\circ\phi_\lambda$.

To prove part (c), suppose that $\phi\in\aut$ and $\phi(x)=x$. We know already that $\phi(h)=ah+b$ and $\phi(y)=cy$, for some $a, b, c\in\CC$ with $a, c\neq 0$. Then $xf(ah+b)=x\phi(f(h))=\phi(h)x=(ah+b)x=x(af(h)+b)$, and we obtain
\begin{equation}\label{E:aut:H:f}
f(ah+b)=af(h)+b.
\end{equation}
Therefore, 
\begin{align*}
 c(f(h)-h) &=c[y, x]=\phi([y, x])=\phi(f(h)-h)=af(h)+b-(ah+b)\\&=a(f(h)-h),
\end{align*}
and we conclude that $c=a$.

Write $f(h)=\sum_{k=0}^{n}a_k h^k$, where $n=\degg f$ and $a_k\in\CC$. Applying the derivation $\frac{d}{dh}$ to \eqref{E:aut:H:f} $n-1$ times yields $a^{n-1}f^{(n-1)}(ah+b)=af^{(n-1)}(h)$, as $n-1\geq 1$. As $f^{(n-1)}(h)=(n-1)!(na_n h+a_{n-1})$, we obtain
\begin{equation}\label{E:aut:H:ab}
a^{n-1}=1 \quad \mbox{and}\quad b=\frac{(a-1)a_{n-1}}{na_n}.
\end{equation}

Let $\mathsf{U}_{n-1}=\{ \zeta\in\CC^*\mid \zeta^{n-1}=1 \}$ be the cyclic group of order $n-1$, and define a map
\begin{equation*}
\Psi: \{\phi\in\aut \mid \phi(x)=x\}\longrightarrow \mathsf{U}_{n-1}, \quad \phi\mapsto a,\ \ \mbox{where $\phi(h)=ah+b$.}
\end{equation*}
Then $\Psi$ is well defined by \eqref{E:aut:H:ab}, and it is a group homomorphism. If $\Psi(\phi)=1$ for some $\phi\in\aut$ with $\phi(x)=x$, then the above shows that $\phi(y)=y$ and $\phi(h)=h+b$. Again by \eqref{E:aut:H:ab} we deduce that $b=0$, so $\phi$ is the identity on $\Hf$. This shows that $\Psi$ is an injective group homomorphism and thus $\{\phi\in\aut \mid \phi(x)=x\}$ is isomorphic to a subgroup of $\mathsf{U}_{n-1}$; hence it is a finite cyclic group whose order divides $n-1$.
\end{proof}

It is now an easy matter to determine the structure of $\aut$. The symbol $\dot{\times}$ used below denotes the internal direct product of subgroups of a group.

\begin{thm}\label{T:aut:main}
Assume $\degg f>1$. Then 
\begin{equation}\label{E:aut:main}
\aut=\{\phi_\lambda\mid \lambda\in\CC^*\}\dot{\times}\{\phi\in\aut \mid \phi(x)=x\}
\end{equation}
is an abelian group, where:
\begin{itemize}
\item $\{\phi_\lambda\mid \lambda\in\CC^*\}\cong\CC^*$ and $\phi_\lambda$ is defined in \eqref{E:aut:defpl};
\item $\{\phi\in\aut \mid \phi(x)=x\}$ is a finite cyclic group whose order divides $(\degg f)-1$ and which, as a set, can be identified with $\{ (a, b)\in\CC^*\times\CC \mid f(ah+b)=af(h)+b \}$ via the correspondence $\phi\mapsto (a, b)$, where $\phi(h)=ah+b$.
\end{itemize}
\end{thm}
\begin{proof}
Since we have already seen in Proposition~\ref{P:aut:H} that $\{\phi_\lambda\mid \lambda\in\CC^*\}$ is central, in order to prove the direct product decomposition in \eqref{E:aut:main}, it remains to show that the two subgroups have trivial intersection, which is clear, and generate $\aut$. Let $\psi\in\aut$. Then there is $\lambda\in\CC^*$ such that $\psi(x)=\lambda x$, whence $\phi^{-1}_\lambda\circ\psi\in\{\phi\in\aut \mid \phi(x)=x\}$, and this shows the latter claim. Moreover, since $\{\phi\in\aut \mid \phi(x)=x\}$ is abelian, by Proposition~\ref{P:aut:H}(c), the group $\aut$ must also be abelian.

The remaining parts of the theorem have already been proved, except for the observation that $\{\phi\in\aut \mid \phi(x)=x\}$ can be identified with the set $\{ (a, b)\in\CC^*\times\CC \mid f(ah+b)=af(h)+b \}$. Indeed, if $\phi(x)=x$, then we have seen in the proof of Proposition~\ref{P:aut:H} that $\phi(h)=ah+b$ and $\phi(y)=ay$, for some $a, b\in\CC$ with $a\neq 0$, and \eqref{E:aut:H:f} must hold. This shows that the correspondence $\phi\mapsto (a, b)$ is well defined and one-to-one. Conversely, given $(a, b)\in\CC^*\times\CC$ satisfying $f(ah+b)=af(h)+b$, it is routine to check that there is an automorphism of $\Hf$ defined by the conditions $\phi(x)=x$, $\phi(y)=ay$, $\phi(h)=ah+b$, and this shows the correspondence is onto.
\end{proof}

\begin{remark}
Any pair $(a, b)\in\CC^*\times\CC$ satisfying $f(ah+b)=af(h)+b$ must also satisfy \eqref{E:aut:H:ab}, where $n=\degg f$, although the conditions in \eqref{E:aut:H:ab} are not sufficient (see the examples below). Thus, for each $(n-1)$-th root of unity $a$, the corresponding scalar $b$ is determined by \eqref{E:aut:H:ab}, but one still needs to check the relation $f(ah+b)=af(h)+b$ for the pair $(a, b)$.
\end{remark}

\begin{exams}\hfill
\begin{itemize}
\item [(a)] If $\degg f=2$, then $n=2$ in \eqref{E:aut:H:ab}, so $a=1$ and $b=0$, and the pair $(1, 0)$ corresponds to the identity map. It follows that the group $\{\phi\in\aut \mid \phi(x)=x\}$ is trivial and $\aut\cong\CC^*$.
\item [(b)] Let $f(h)=h^3+h$. Then $n=3$ in \eqref{E:aut:H:ab}, so $a=\pm1$. If $a=1$, then $b=0$, and the corresponding automorphism is the identity. If $a=-1$, then $b=0$, and in fact $f(-h)=-f(h)$. Therefore there is an automorphism $\phi$ of $\Hf$ such that $\phi(x)=x$, $\phi(y)=-y$, $\phi(h)=-h$, and $\aut\cong\CC^*\times\ZZ_2$.
\item [(c)] Let $f(h)=h^3+h+1$. Then $n=3$ in \eqref{E:aut:H:ab}, so $a=\pm1$. If $a=-1$, then \eqref{E:aut:H:ab} yields $b=0$, but $f(-h)\neq-f(h)$, so the group $\{\phi\in\aut \mid \phi(x)=x\}$ is trivial and $\aut\cong\CC^*$.
\item [(d)] Let $f(h)=h^n$, for $n>1$. Then \eqref{E:aut:H:ab} says that $a$ is a $(n-1)$-th root of unity and $b=0$. Moreover, $f(ah)=a^nf(h)=af(h)$ for any $(n-1)$-th root of unity $a$. Hence $\aut\cong\CC^*\times\ZZ_{n-1}$.
\item [(e)] Let $f(h)=h^n+1$, for $n>1$. Then, as before, $a$ is a $(n-1)$-th root of unity and $b=0$. However, in this case, $f(ah)=ah^n+1$ whereas $af(h)=ah^n+a$, so equality holds if and only if $a=1$. Hence $\aut\cong\CC^*$.
\item [(f)] Let $f(h)=h^n+h^{k+1}$, for some $n\geq 4$, and take any $1\leq k< n-1$ such that $k\mid n-1$. Then $a$ is a $(n-1)$-th root of unity and $b=0$. In this case $f(ah)=ah^n+a^{k+1}h^{k+1}$ whereas $af(h)= ah^n+ah^{k+1} $, so equality holds if and only if $a^k=1$. By the hypothesis that $k\mid n-1$, we deduce that $\aut\cong\CC^*\times\ZZ_{k}$.
\end{itemize}
\end{exams}

\def\cprime{$'$} \def\cprime{$'$} \def\cprime{$'$}

\noindent \textsc{Samuel A. Lopes} \\
\textit{\small CMUP, Faculdade de Ci\^encias, Universidade do Porto\\
Rua do Campo Alegre 687\\ 
4169-007 Porto, Portugal}\\
\texttt{slopes@fc.up.pt}

\end{document}